\documentclass[10pt]{amsart}
\usepackage{ulem,multicol}
\usepackage{array,youngtab}
\usepackage{color, graphicx, comment}
\usepackage{tikz}
\usepackage{ulem}
\usetikzlibrary{automata,positioning}
\graphicspath{{./pics/}}

\usepackage{mathtools}

\definecolor{darkgreen}{rgb}{0.,0.5,0.}

\DeclareMathOperator{\val}{val}

\DeclareMathOperator{\CF}{CF}

\DeclareMathOperator{\ff}{\mathbb{F}}

\DeclareMathOperator{\ve}{\varepsilon}
\DeclareMathOperator{\tr}{tr}

\allowdisplaybreaks

\numberwithin{equation}{section} \overfullrule 5pt
\newtheorem{thm}{Theorem}[section]
\newtheorem{coro}[thm]{Corollary}
\newtheorem{prop}[thm]{Proposition}

\newtheorem{lem}[thm]{Lemma}

\theoremstyle{definition}

\newtheorem{rk}{Remark}[section]
\newtheorem{eg}{Example}

\title{Algebraic automatic continued fractions  in characteristic $2$ II}
\date{}
\author{Yining Hu}
\address[Yining Hu]{School of Mathematics and Statistics, 
Huazhong University of Science and Technology, Wuhan, PR China}
\email{huyining@protonmail.com}
\thanks{
 This work is supported by the National Science Foundation
of China grant 12001216.
}
\begin{document}
\maketitle

%this comes from the draft note2.tex

\begin{abstract}
	We present two families of automatic sequences that define algebraic continued
	fractions in characteristic $2$. The period-doubling sequence belongs
	to the first first family $\mathcal{P}$; and its sum, the Thue-Morse 
	sequence, belongs to the second family $\mathcal{G}$.  The family 
	$\mathcal{G}$ contains all the iterated sums of 
	sequences from the $\mathcal{P}$ and more.
\end{abstract}

\section{Introduction}
%<<<
In \cite{Hu2022H}, the author and Han considered continued fractions defined
by the Thue-Morse and the period-doubling sequence and proved their algebraicity
in characteristic $2$ for some cases.
In \cite{Bugeaud2022H}, Bugeaud and Han proved the algebraicity of Thue-Morse
continued fractions. In \cite{Hu2022L}, the author and Lasjaunias proved the
algebraicity of period-doubling continued fractions. In \cite{Hu2022}, the
author generalized \cite{Hu2022L}. In this article we give a generalization 
of all these results.

Let $\sigma$ denote the operator that maps a binary sequence $(u_n)_n$ to 
$(\sum_{j=0}^{n-1} u_j \pmod{2})_n$. 
Consider the period-doubling sequence 
$$\mathbf{p}=1,0,1,1,1,0,1,0,...$$
defined as the fixed point of the substitution $1\mapsto 10$, $0\mapsto 11$,
then
$$\sigma(\mathbf{p})= 0,1,1,0,1,0,0,1, \ldots$$ 
is the Thue-Morse sequence $\mathbf{t}$.
Therefore one way to generalize \cite{Hu2022H}, \cite{Bugeaud2022H} and
\cite{Hu2022L} is to consider continued fractions defined by 
$\sigma^n(\mathbf {p})$ for $n\geq 2$. 

On the other hand, in \cite{Hu2022}, the author considered the following 
generalization of the period-doubling sequence:
	 Given an ultimately periodic sequence $\boldsymbol \ve$, starting from the 
	 empty word $W_0$, define the sequence of words 
	 $(W_n)_{n\geq0}$ by setting $W_{n+1}=W_n,\varepsilon_n,W_n$ (note the comma 
	 is for concatenation and it will be omitted when it is suitable).  
	 It is proved
	 in \cite{Hu2022} that $\lim_n W_n$ then defines an algebraic continued
	 fraction in characteristic $2$.
Here we consider a more general case by allowing $W_0$ to be a non-empty finite
word.  We may then assume $\boldsymbol\ve$ to be periodic, because the preperiodic 
part can be absorbed by $W_0$. We let $\mathcal{P}(W_0, \boldsymbol\ve)$ denote
$\lim_n W_n$ and let $\mathcal{P}$ denote the family of such sequences.
The period-doubling sequence $\bf{p}$ is generated in this way by the empty word
and $\boldsymbol \ve =(10)^\infty$. 

Naturally we also want to consider sequences of the form $\sigma^n(\mathbf{u})$
where $n\geq 1$ and $\mathbf{u}\in \mathcal{P}$. It turns out more convenient
to consider yet a larger family $\mathcal{G}$ defined as follows:
Let $\boldsymbol\Upsilon$ be a periodic $(0,1)$-sequence. Let $u_0$ and $v_0$ be
two finite words. For $n\geq 0$, define
%\begin{align*}
%	u_{n+1} &= \begin{cases}u_n,u_n& \upsilon =1 \\ u_n,v_n&\upsilon=0\end{cases}\\
%	v_{n+1} &= \begin{cases}v_n,v_n& \upsilon =1 \\ v_n,u_n&\upsilon=0\end{cases}\\
%\end{align*}
\begin{align*}
	u_{n+1} &= u_n, u_n\\
	v_{n+1} &= v_n, v_n
\end{align*}
if $\Upsilon_n=0$, and
\begin{align*}
	u_{n+1} &= u_n, v_n\\
	v_{n+1} &= v_n, u_n
\end{align*}
otherwise.
We let $\mathcal{G}(u_0,v_0, \boldsymbol \Upsilon)$ denote $\lim_n u_n$ and let 
$\mathcal{G}$ denote the family of such sequences.

In Section \ref{sec:pd2} we prove the algebraicity of continued fractions
defined by sequences from $\mathcal{P}$. A simpler proof
is given  in \cite{Hu2022} for the case where $W_0$ is the empty word.
In Section \ref{sec:pd2sum} we first give two examples, then prove the 
algebraicity of continued fractions defined by sequences from $\mathcal{G}$.
Finally we prove that for all $n\geq 1$ and all
binary sequences $\mathbf{u}\in \mathcal{P}$, $\sigma^n(\mathbf{u})\in \mathcal{G}$.
%>>>

\subsection{Continued Fractions in Power Series Rings}\label{ss:1}
%<<<
We adopt the definition of continued fractions and notion of convergence 
as in \cite{Hu2022}, which we recall in this subsection for convenience.

Let $A=\{a_0,\ldots, a_k\}$ be a finite alphabet. We treat $a_j$ as formal
variables and define
\begin{align*}
	\ff_2[A]:=\ff_2[a_0,\ldots, a_k]\\
	\ff_2(A):=\ff_2(a_0,\ldots, a_k)\\
	\ff_2((A)):=\ff_2((\frac{1}{a_0},\ldots, \frac{1}{a_k})).
\end{align*}
Here	$\ff_2((\frac{1}{a_0},\ldots, \frac{1}{a_k}))$ denotes the ring of
power series of the form
\begin{equation}
\varphi=\sum_{n_0, \ldots, n_k \geq N} c_{n_0,n_1,\ldots, n_k} a_0^{-n_0}\cdots {a_k}^{-n_k},\label{eq:star}\end{equation}
where $N$ is an integer and $c_{n_0,n_1,\ldots, n_k}\in \ff_2$.

Let $(u_n)_{n\geq 0}$ be a sequence taking values in $A$. It defines a
formal power series $\sum_{n\geq 0} u_n z^n$ in $\ff_2[A][[z]]$.

We define a norm on $\ff_2((A))$ by assigning a series of the form 
\eqref{eq:star}
the number $2^{-m}$, where $m=\min\{ n_0+n_1 \cdots+ n_k \mid \ c_{n_0,n_1,
\ldots, n_k}\neq 0\}$ (with the convention that $\min\emptyset=\infty$)
is the {\it valuation} of the series $\varphi$ and is denoted by $\val(\varphi)$.
This norm makes  $\ff_2((A))$ an ultrametric space.

The continued fraction $\CF(\mathbf{u})=[u_0,u_1,\ldots]$ is defined as 
the limit of the sequence $([u_0,u_1,\ldots, u_n])_n$:
\begin{align*}
	[u_0]&=u_0,\\
	[u_0, u_1, \ldots, u_{n}]&=u_0+\frac{1}{[u_1,\ldots, u_n]}\in \ff_2((A)),
\end{align*}
for $n\geq 1$. 
For example, 
\begin{align*}
	[u_0,u_1,u_2]&=u_0+\cfrac{1}{u_1+\frac{1}{u_2}}\\
	&=u_0+\frac{u_1^{-1}}{1+(u_1u_2)^{-1}}\\
	&=u_0+u_1^{-1}+u_1^{-2}u_2^{-1}+u_1^{-3}u_2^{-2}+\cdots
\end{align*}
Define
\begin{equation}\label{eq:cfm}
M_n=
\begin{pmatrix} 1& \frac{1}{u_{n}}\\ \frac{1}{u_{n}}&0\end{pmatrix}
\begin{pmatrix} 1& \frac{1}{u_{n-1}}\\ \frac{1}{u_{n-1}}&0\end{pmatrix}
	\cdots
\begin{pmatrix} 1& \frac{1}{u_{0}}\\ \frac{1}{u_{0}}&0\end{pmatrix}
\end{equation}
	then 
	\begin{equation}\label{eq:cfmn}
	 [u_0, u_1, \ldots, u_{n}] =\frac{M_{n,0,1}}{M_{n,0,0}}.
	\end{equation}
	In general, we do not have the convergence of $(M_{n,0,1})_n$ and
  $(M_{n,0,0})_n$, but we do have the convergence of 
$\left(\frac{M_{n,0,1}}{M_{n,0,0}}\right)_n$, which is proved in the same way
as in the case of classical contineud fraction for real numbers.

We may also choose to specialize the letters in $A$ to non-constant polynomials
in $\ff_2[z]$ (this gives a continuous map from $\ff_2((A))$ to $\ff_2((1/z))$),
and then $\CF({\mathbf{u}})$ can be seen as formal power series in $\ff_2((1/z))$.
	 For basic information on continued fractions, particularly in power series rings, the reader is refered to \cite{Lasjaunias2017}.

%>>>

\subsection{Main results}\label{ssec:thm}
%<<<
Let $\bf{u}$ be a sequence taking values in a finite alphabet $A$.
\begin{thm}\label{thm:pd2}
	If $\mathbf{u}=\mathcal{P}(W_0, \boldsymbol\ve)$ is a sequence in $\mathcal{P}$, 
	where $\boldsymbol \ve$ has period $n$, then $\CF(\mathbf{u})$ is algebraic
	over $\ff_2({A})$ of degree at most $2^n$.
\end{thm}

\begin{thm}\label{thm:pd2sum}
	If $\mathbf{u}=\mathcal{G}(u_0,v_0, \boldsymbol\Upsilon)$ is a sequence in 
	$\mathcal{G}$, where $\boldsymbol \Upsilon$ has period $k$, and contains an
	even number of $1$'s in one period, then $\CF(\mathbf{u})$ is algebraic
	over $\ff_2({A})$ of degree at most $2^k$.
\end{thm}

\begin{coro}\label{thm:coro}
	Let $\mathbf{u}=\mathcal{P}(W_0, \boldsymbol\ve)$ be a binary sequence in 
	$\mathcal{P}$, where $\boldsymbol \ve$ has period $n$.
	Then for all $k\geq 1$, 
	$\CF(\sigma^k(\mathbf{u}))$ is 
	algebraic over $\ff_2({A})$ of degree at most $2^n$.
\end{coro}
\begin{rk}
	When we apply $\sigma$ to a binary sequence, we regard its terms as $0$'s
	and $1$'s; when we consider the continued fraction that it defines, we
	regard its terms as formal variables.
\end{rk}
\begin{rk}
	In inverse of $\sigma$ maps a sequence $(u_n)_n$ to 
	$(u_{n+1}-u_n)_n$. 
	Then $\mathbf{p}=\sigma^{-1}(\mathbf{t})$. We could continue to consider
	$$\sigma^{-1}(\mathbf{p})=1,1,0,0,1,1,1,1,1,1,0,0,1,1,0,0,\ldots$$
	However, the continued fraction defined by $\sigma^{-1}
	(\mathbf{p})$ does not seem to be algebraic. 
	The sequence $\sigma^{-1}(\mathbf{p})$ can also be obtained by replacing 
	$1$ by $11$ and $0$ by $00$ in $\mathbf{p}$.  On the other hand, by 
	Theorem \ref{thm:pd2sum}, if we replace $0$ and $1$ in the Thue-Morse
	sequence by finite words $u_0$ and $v_0$, we always obtain a sequence that 
	defines an algebraic continued fraction.
\end{rk}
%>>>

\section{The family $\mathcal{P}$} \label{sec:pd2}
%<<<
In this section we prove Theorem \ref{thm:pd2}.

Let $W_0$ be a finite word and let $\boldsymbol\ve$ be a periodic sequence
of period $n$, 
both taking values in a finite alphabet $A$.

Let $l$ be the length of $W_0$. By \eqref{eq:cfm} and \eqref{eq:cfmn}, 
if we define
\begin{equation}
m_0=
\begin{pmatrix} 1& \frac{1}{W_{l-1}}\\ \frac{1}{W_{l-1}}&0\end{pmatrix}
\begin{pmatrix} 1& \frac{1}{W_{l-2}}\\ \frac{1}{W_{l-2}}&0\end{pmatrix}
	\cdots
\begin{pmatrix} 1& \frac{1}{W_{0}}\\ \frac{1}{W_{0}}&0\end{pmatrix}
\end{equation}
and 
	$$m_{n+1}=m_n \begin{pmatrix} 1& 1/\ve_n \\1/\ve_n&0\end{pmatrix} m_n,$$
		Then
		$$\CF(\mathcal{P}(W_0,\boldsymbol\ve))=\lim_n \frac{m_{n,0,1}}{m_{n,0,0}}.$$

Define
%$$m_0=\begin{pmatrix} a_0 & a_1 \\a_2&a_3\end{pmatrix}$$
	for $n\geq 0$,
$$b_n=\begin{pmatrix} 0& 1/\ve_n \\1/\ve_n&1\end{pmatrix},$$
	$$d_n=\det(m_n),$$
$$	l_n=(m_{n,0,1}+m_{n,1,0})/\varepsilon_n +m_{n,0,0}.$$
Define $L_0=1$, and $L_{n+1}=L_n \cdot l_n$ for $n\geq 0$.

Lemma \ref{pd2:lem1} and \ref{pd2:lem2} does not use the structure of $m_0$,
and is true for any generic $2\times 2$ matrix in place of $m_0$.
Lemma \ref{pd2:lem1} does not use the periodicity of $\boldsymbol\ve$.
	\begin{lem}\label{pd2:lem1}
		For $n\geq 1$,
		\begin{equation}\label{eq:lem1}
			m_n=L_n\cdot (m_0 + d_0 b_0/L_1 + d_1b_1/L_2\cdots + d_{n-1} b_{n-1}/L_n) 
		\end{equation}
	\end{lem}
	%<<<
	\begin{proof}
		By induction. The formula is true for $n=1$:
		$$m_1=L_1 m_0 +d_0b_0.$$
		Suppose the formula is true 
		for $n$. 
		Let $m_0'= m_1$, $\ve_n'=\ve_{n+1}$, and define $m_n'$, etc.  Then $l'_n
		=l_{n+1}$, $d'_n=d_{n+1}$, $b'_n=b_{n+1}$, and $L_n'= L_n \cdot l_n/l_0
		=L_{n+1}/l_0=L_{n+1}/L_1$.
    By the induction hypothesis,
		\begin{align*}
			&\quad \; m_{n+1}\\
			&=m_{n}'\\
			&=L_n'\cdot (m_0' + d_0' b_0'/L_1' + d_1' b_1' /L_2' \cdots + d_{n-1}' b_{n-1}'/L_n')      \\
%		&=L_n'\cdot (m_1 + d_1 b_1/l_0' + d_2 b_2 /l_0'/l_1' \cdots + d_{n} b_{n}/L_n')      \\
			&=L_n'\cdot (m_1 + d_1 b_1/L_1' + d_2 b_2 /L_2' \cdots + 
			d_{n} b_{n}/L_n')      \\
			&=L_n'\cdot (L_1\cdot m_0+ d_0b_0 + d_1 b_1/L_1' + d_2 b_2 /L_2' \cdots + 
			d_{n} b_{n}/L_n')      \\
			&=L_{n+1}\cdot (m_0+ d_0b_0/L_1 + d_1 b_1/L_2 + d_2 b_2 /L_3 \cdots + d_{n} b_{n}/L_{n+1})      \qedhere
		\end{align*}
%	\begin{align*}
%		&\quad \; m_{n+1}\\
%		&=m'_n\\
%		&=L_n'\cdot (m_0'+\sum_{j=0}^{n-1} d_j'b_j'/L_{j+1}')\\
%		&=L_n'\cdot (L_1\cdot m_0+d_0b_0+\sum_{j=0}^{n-1} d_{j+1}b_{j+1}/L_{j+1}')\\
%		&=L_{n+1}\cdot (m_0+d_0b_0/L_1+\sum_{j=0}^{n-1} d_{j+1}b_{j+1}/L_{j+2})\\
%		&=L_{n+1}\cdot (m_0+\sum_{j=0}^{n} d_{j}b_{j}/L_{j+1})
%	\end{align*}
	\end{proof}
	%>>>

	\begin{lem}\label{pd2:lem2}
		i) For all $j\geq 0$, 
		$$l_{n+j}=L_n^{2^j}l_j.$$
		ii) For all $j\geq 0$, 
		\begin{equation*}
		L_{n+j}=L_n^{2^j} L_j.
		\end{equation*}
	\end{lem}
	%<<<
	\begin{proof}
		i) By induction.
		In \eqref{eq:lem1}, the only term that
		contributes to $l_n$ is $L_n\cdot m_0$, and
		\begin{align*}
			l_n&= (m_{n,0,1}+m_{n,1,0})/\varepsilon_n +m_{n,0,0}\\
			&= L_n(m_{1,0,1}+m_{1,1,0})/\varepsilon_0 +L_n m_{0,0,0}\\
			&=L_n l_0.
		\end{align*}
		Suppose that
		$$l_{n+j}=L_n^{2^j}l_j.$$
		As in the proof of Lemma \ref{pd2:lem1}, let $m'_0=m_1$ and 
		$\varepsilon_n'=\ve_{n+1}$, then
			$$
				l_{n+j+1}=l_{n+j}'=L_n'^{2^j}l_j'=(L_n\cdot l_n/l_0)^{2^j} l_{j+1}
				=(L_n\cdot L_n l_0/l_0 )^{2^{j}}l_{j+1}
				=L_n^{2^{j+1}}l_{j+1}.
			$$
		ii) By induction. For $j=0$, the identity
		$$L_n=L_n^{2^0}L_0$$
		holds trivially.

		Suppose
		$$L_{n+j}=L_n^{2^j} L_j.$$
		Then
		\begin{equation*}
		L_{n+j+1}=L_{n+j}l_{n+j}=L_n^{2^j}L_jL_n^{2^j}l_j=L_n^{2^{j+1}}L_{j+1}.
		\qedhere
		\end{equation*}
	\end{proof}
	%>>>

	\begin{lem}\label{pd2:lem3}
		The sequence $(m_{kn})_k$ converges.
	\end{lem}
	%<<<
	\begin{proof}
		By Lemma \ref{pd2:lem1},
	$$m_{kn}=L_{kn}\cdot (m_0 + d_0b_0/L_1+\cdots + d_{kn-1}b_{kn-1}/L_{kn}).$$
		To prove the convergence of $(m_{kn})_k$, we only need to prove the 
		convergence of both factors on the right.
    It is easy to prove by induction that for all $k\geq 0$,
		$\val(m_{k,0,0})=0$ and 
		$\val(m_{k,i,j})>0$ for $(i,j)\neq (0,0)$.  From this we deduce that
		$\val(l_k)=0$ for all $k\geq0$, and therefore $\val(L_n)=0$ and
		$\val(1+L_n)\geq 1$.
		By Lemma \ref{pd2:lem2}, 
		$$L_{(k+1)n}=L_{n}^{2^{kn}}L_{kn}$$
		and therefore
		$$\val\left(L_{(k+1)n}-L_{kn}\right)=\val\left(\left(1-L_n^{2^{kn}}\right)
		\cdot L_{kn} \right) \geq 2^{kn}.$$
		This proves the convergence of $(L_{kn})_k$.

		On the other hand, it is easy to prove that $\val(d_j)\geq 2^{2j}$. 
		Therefore $$\val(d_j/L_j)\geq 2^{2j}.$$
		This proves the convergence of 
		$(m_0 + d_0b_0/L_1+\cdots + d_{kn-1}b_{kn-1}/L_{kn})_k$.
	\end{proof}
	%>>>
	%<<<
	\begin{proof}[Proof of Theorem \ref{thm:pd2}]
	Let
		$$f=\lim_{k\rightarrow \infty} L_{kn}
		=\lim_{k\rightarrow\infty} L_n^{2^0+2^n+\cdots 2^{(k-1)n}}.$$
		Then $f$ is algebraic of degree at most $2^n-1$:
	$$f^{2^n} =\lim_{k\rightarrow\infty} L_n^{2^n+2^{2n}+\cdots 2^{kn}} = f/L_n.$$
		By Lemma \ref{pd2:lem2} and taking the limit, we have 
	$$
	\lim_{k\rightarrow \infty} m_{kn}=f\cdot (m_0 +d_0b_0/L_1 +d_1b_1/L_2+\cdots).
	$$
	For $j=0,1,\ldots, n-1$, define 
		$$H_j=\sum_{k=0}^\infty d_{kn+j}/L_{kn+j+1},$$
	then 
		\begin{equation}\label{eq:pd2lim}
	\lim_{k\rightarrow \infty} m_{kn}=f\cdot (m_0 + b_0 H_0 + b_1 H_1 +\cdots +
	b_{n-1} H_{n-1})
		\end{equation}
	We prove easily by induction that 
	$d_{(k+1)n}=d_{kn}^{2^n}\cdot \lambda$ where 
		$\lambda=1/(\ve_0^{2^n}\cdots \ve_{n-1}^{2^1})$. Using this and Lemma \ref{pd2:lem2}, 
	we obtain the algebraicity of $H_0$:
		\begin{equation}\label{eq:pd2H}
	H_0^{2^n}\cdot \lambda/L_1= \frac{d_0^{2^n}\cdot \lambda}{L_1^{2^n} L_1}+
	\frac{d_n^{2^n}\cdot \lambda}{L_{n+1}^{2^n} L_1}+\cdots 
	=\frac{d_n}{L_{n+1}}+ \frac{d_{2n}}{L_{2n+1}}+\cdots=H_0+\frac{d_0}{L_1}.
		\end{equation}
	On the other hand, for $j=1, \ldots, n-1$, 
	$d_{kn+j}/d_{kn}^{2^j}=d_j/d_0^{2^j}.$
	Also, by Lemma \ref{pd2:lem2}, 
	$$\frac{L_{n+j+1}}{L_{n+1}^{2^j}}=
	\frac{L_n^{2^{j+1}}L_{j+1}}{(L_n^2L_{1})^{2^j}}=
	\frac{L_{j+1}}{L_1^{2^j}}, $$
	and by induction,
	$$\frac{L_{kn+j+1}}{L_{kn+1}^{2^j}}=
	\frac{L_{j+1}}{L_1^{2^j}}. $$
	Therefore 
		\begin{equation}\label{eq:pd2Hj}
	H_j=H_0^{2^j}\cdot \frac{d_j}{d_0^{2^j}} \cdot \frac{L_{j+1}}{L_1^{2^j}}.
		\end{equation}
		%This show that every component of $\lim_{k}m_{kn}$ lives in $\ff_2(A)[f, H_0]$, and therefore are algebraic.
		By \eqref{eq:pd2lim}, \eqref{eq:pd2H}, and \eqref{eq:pd2Hj}, the continued
		fraction
		$$\CF(\mathcal{P}(W_0,\boldsymbol\ve))=\lim_{k\rightarrow \infty} \frac{m_{kn,0,1}}{m_{kn,0,0}}$$
		lives in $\ff_2(A)[H_0]$ and is algebraic of degree at most $2^n$.
	\end{proof}
	%>>>
	%>>>

\section{The family $\mathcal{G}$}\label{sec:pd2sum}

\subsection{Examples}
%<<<
We give two examples before proving Theorem \ref{thm:pd2sum} (we do not need
the theorem to prove the algebraicity of each example),
so that the reader may get a quick idea of the proof.
%<<<
\begin{eg}
	Let $u_0$ and $v_0$ be single letters $a$ and $b$. Let $\boldsymbol \Upsilon=
	1^\infty$. Then $\mathcal{G}(u_0,v_0,\boldsymbol\Upsilon)$ is the Thue-Morse
	sequence $\mathbf{t}=abbabaab\ldots$ This example was first proved in 
	\cite{Bugeaud2022H}.
	Define 
	$$m_0=\begin{pmatrix}1&1/a\\1/a&0\end{pmatrix}, \quad w_0=
		\begin{pmatrix}1&1/b\\1/b&0\end{pmatrix}
		$$
		and for $k\geq 0$,
$$m_{k+1}=w_{k}m_{k}\quad w_{k+1}=m_{k}w_{k}.$$	
By \eqref{eq:cfm} and \eqref{eq:cfmn}, 
	$$\CF(\mathbf{t})=\lim_{k} \frac{m_{2k+1,0,1}}{m_{2k+1,0,0}}.$$
	We will show that every entry of $(m_{2k+1})_k$ converges, and the limit is 
	algebraic.
	The key to the proof is the relation
	\begin{equation}\label{eq:tmkey}
		m_{3}=(m_1 + d/co +d^2/r/co^2 )\cdot l,\tag{$\star$}
	\end{equation}
	where 
	$$d=\det(m_1),\; r=\tr(m_1),\; co=m_1+w_1+r,\; l=r\cdot co^2.$$
	In fact, the relation \eqref{eq:tmkey} holds regardless of the particular 
	form of $m_0$ and $w_0$; it remains true if we had taken $m_0$ and $w_0$ 
	to be any $2\times 2$ matrix in characteristic $2$. 
	Therefore, it actually gives us the relation between
	$m_{2k+1}$ and $m_{2k+3}$ for all $k\geq 0$. For example,
	if we define $m_0'=m_2$, and $w_0'=w_2$, and 
	define $m'_k$, $w'_k$, $d'$, $r'$, $co'$, $l'$ accordingly, 
%	$$m'_{2k+1}=m_{2k}'^2\quad w_{2k+1}'=w_{2k}'^2 $$
%$$m'_{2k+2}=w_{2k+1}'m'_{2k+1}\quad w'_{2k+2}=m'_{2k+1}w'_{2k+1},$$	
%	$$d'=\det(m'_1),\; r'=\tr(m'_1),\; co'=m'_1+w'_1+r',\; l'=r'\cdot co'^2,$$
	then by \eqref{eq:tmkey},
	$$m_{3}'=(m_1' + d'/co' +d'^2/r'/co'^2 )\cdot l'.$$
	We verify directly
	$$d'=d^4,\; r'=r\cdot l,\; co'=co\cdot l,\; l'=l^4.$$
	Therefore 
	\begin{align*}
		&\quad \;m_{5}=m_{3}'=(m_3 + d^4/co/l +d^8/r/co^2/l^3 )\cdot l^4\\
		&=l^{1+4}\cdot (m_1+  (d+d^4/l^2)/co + (d^2+d^8/l^4)/r/co^2)
	\end{align*}
	We continue in this way to find
	$$
		m_7=l^{1+4+16}\cdot (m_1+  (d+d^4/l^2+d^4/l^{2+8})/co + 
		(d^2+d^8/l^4+d^{32}/l^{4+16})/r/co^2),
	$$
	etc.
	Define 
	$$f=l^{1+4+16+64+\cdots}$$
	$$H=d+d^4/l^2+d^{16}/l^{2+8}+d^{64}/l^{2+8+32}+\cdots.$$
	It is easy to prove by induction that
	$$ \lim_{k}m_{2k+1}= f\cdot (m_1+H/co+H^2/r/co^2). $$
	Both $f$ and  $H$ are algebraic:
	$$f^4=f/l$$
	$$H^4/l^2=H+d.$$
	Therefore $\CF(\mathbf{t})\in \ff_2(a,b)[H]$, and is algebraic of degree 
	at most $4$.
\end{eg}
%>>>

%<<<
\begin{eg}
	Let $u_0=a$, $v_0=b$, $\boldsymbol\Upsilon=(011)^\infty$, 
	$\mathbf{u}=\mathcal{G}(u_0,v_0,\boldsymbol\Upsilon)$.
Define 
	$$m_0=\begin{pmatrix}1&1/a\\1/a&0\end{pmatrix}^2, \quad w_0=
		\begin{pmatrix}1&1/b\\1/b&0\end{pmatrix}^2,
		$$
		and for $k\geq 0$,
$$m_{3k+1}=w_{3k}m_{3k}\quad w_{3k+1}=m_{3k}w_{3k}$$	
$$m_{3k+2}=w_{3k+1}m_{3k+1}\quad w_{3k+2}=m_{3k+1}w_{3k+1}$$	
		$$m_{3k+3}=m_{3k+2}^2\quad w_{3k+3}=w_{3k+2}^2. $$
By \eqref{eq:cfm} and \eqref{eq:cfmn}, 
	$$\CF(\mathbf{u})=\lim_{k} \frac{m_{3k+1,0,1}}{m_{3k+1,0,0}}.$$
	We will show that every entry of $(m_{3k+1})_k$ converges, and the limit is 
	algebraic.
	The key to the proof is the relation
	\begin{equation}\label{eq:eg2key}
		m_4=(m_1+d/co+d^2/co^3+d^4/r/co^6)\cdot l.\tag{$\diamondsuit$}
	\end{equation}
	where 
	$$d=\det(m_1),\; r=\tr(m_1),\; co=m_1+w_1+r,\; l=r\cdot co^6.$$
	As in the previous example, the relation \eqref{eq:eg2key} 
	holds regardless of the particular form of $m_0$ and $w_0$.
%Using the relation \eqref{eq:eg2key}, we can deduce the relation between
%$m_{3k+1}$ and $m_{3k+4}$ for all $k\geq 0$. For example,
	In particular, if we define $m_0'=m_3$, and $w_0'=w_3$, and 
	define $m'_k$, $w'_k$, $d'$, $r'$, $co'$, $l'$ accordingly, 
	then by \eqref{eq:eg2key},
	$$ m_4'=(m_1'+d'/co'+d'^2/co'^3+d'^4/r'/co'^6)\cdot l'.$$
	We verify directly
	$$d'=d^8,\; r'=r\cdot l,\; co'=co\cdot l,\; l'=l^8.$$
	Therefore 
	\begin{align*}
		&\quad \;m_{7}=m_{4}'=(m_4 + d^8/co/l +d^{16}/co^3/l^3+d^{32}/r/co^6/l^7 )
		\cdot l^8\\
		&=l^{1+8}\cdot(m_1 + (d+d^8/l^2)/co+(d^2+d^{16}/l^4)/co^3+
		(d^4+d^{32}/l^8)/r/co^6).
	\end{align*}
	Define
	$$f=l^{1+8+64+512+\cdots}$$
	$$H=d+d^8/l^2+d^{64}/l^{2+16}+d^{512}/l^{2+16+128}+\cdots.$$
	It is easy to prove by induction that
	$$ \lim_{k}m_{3k+1}= f\cdot (m_1+H/co+H^2/co^3+H^4/r/co^6). $$
	Both $f$ and  $H$ are algebraic:
	$$f^8=f/l$$
	$$H^8/l^2=H+d.$$
	Therefore $\CF(\mathbf{u})\in \ff_2(a,b)[H]$, and is algebraic of degree 
	at most $8$.
\end{eg}
%>>>
	%>>>

\subsection{Proof of Theorem \ref{thm:pd2sum}}
%<<<
In this subsection we prove Theorem \ref{thm:pd2sum}.
In the examples, we found the relations \eqref{eq:tmkey} and \eqref{eq:eg2key}
by hand and verified them by direct computation.
The key of the proof is to find and prove such relations
 for all sequences in $\mathcal{G}$ (this is done in Proposition 
 \ref{prop:main}); the other steps follow naturally.

%<<<
Let $u_0$, $v_0$ be two finite non-empty words. Let $\boldsymbol\Upsilon$ 
be a periodic $(0,1)$-sequence of periodc $k$ and containing an even number of 
$1$'s in one period. By replacing $u_0$ and $v_0$ if necessary, we may assume 
that $\boldsymbol\Upsilon$ begins with $1$ (except for the trivial case
where $\boldsymbol\Upsilon= 0^\infty$)

To simplify notation, we choose to identify scalar matrices with scalars.
Let $m_0$, $w_0$ be two $2\times 2$ matrices with entries in a field of
characteristic $2$. Let $m_1=w_0m_0$, $w_1=m_0w_0$.
We define $m_{1s}$ and $w_{1s}$ for all binary word $s$ inductively as follows:
\begin{align*}
	m_{1s0}&=m_{1s}^2\\
	m_{1s1}&=w_{1s}m_{1s}\\
	w_{1s0}&=w_{1s}^2\\
	w_{1s1}&=m_{1s} w_{1s}.
\end{align*}
Define $d=\det(m_1)$, $r=\tr(m_1)$ (the trace of $m_1$), and $co=m_1+w_1+r$. 
We verify directly the following relations that will be used in the proof
of Proposition \ref{prop:main}:
\begin{align*}
	\det(m_1)&=\det(w_1)\\
	\tr(m_1)&=\tr(w_1)\\
	co^2&=\tr(m_{11})  \quad \text{ ($co^2$ is a scalar matrix)}\\
	co\cdot m_1 + m_1\cdot co&= co\cdot (co+r)\\
	co\cdot w_1 + w_1\cdot co&= co\cdot (co+r)\\
	co\cdot m_1 &=w_1\cdot co\\
	co\cdot w_1 &=m_1\cdot co\\
	m_1^2&=r\cdot m_1+d\\
	w_1^2&=r\cdot w_1+d\\
	w_1\cdot m_1&= w_1\cdot co+d\\
	m_1\cdot w_1&= m_1\cdot co+d.
\end{align*}
For all binary word $ s$, define $t(s)$ as the  sum of digits of $s$ modulo $2$.

Let $s$ be an arbitrary non-empty binary word.
For $j<|s|$, define $s(j)$ to be $s_0\ldots s_{j-1}$, the prefix of length
$j$ of $s$.
Define $e(s)$ inductively as follows: 
If $s$ is the empty word, then $e(s)=0$; define $e(s0)=2e(s)+t(s0)$, 
$e(s1)=2e(s)+t(s1)$. 
Define and $c_j(s) = d^{2^{j-1}}/r^{2^j-1-e(s(j))}/co^{e(s(j))}$.
We write $c_j$ for short when there is no ambiguity.

%>>>

\begin{prop}\label{prop:main}
	Let $s$ be a binary word of length $k\geq 1$.
	If $t(s)=0$, then
\begin{align}
	m_{1s}&=(m_1 +c_1 + \cdots + c_k)\cdot d^{2^{k-1}}/c_k, \label{lem:eg}\\
	w_{1s}&=(w_1 +c_1 + \cdots + c_k)\cdot d^{2^{k-1}}/c_k; \label{lem:eg2}
\end{align}
	if $t(s)=1$, then
\begin{align}
	m_{1s}&=(w_1 +c_1 + \cdots + c_k)\cdot d^{2^{k-1}}/c_k, \label{lem:eg3}\\
	w_{1s}&=(m_1 +c_1 + \cdots + c_k)\cdot d^{2^{k-1}}/c_k.\label{lem:eg4}
\end{align}
\end{prop}
%<<<
\begin{proof}
	We prove \eqref{lem:eg3} by induction on the length of $s$, 
	and obtain \eqref{lem:eg4} by symmetry.
	The proofs of \eqref{lem:eg} and \eqref{lem:eg2} are similar.
	The cases $s=0$ and $s=1$ are verified directly.

	Suppose the \eqref{lem:eg} to \eqref{lem:eg4} are true for $s$ with 
	$k=|s|\geq 1$ and $t(s)=1$.
	In the proof, as in the statement of the lemma, $c_j$ alone stands for 
	$c_j(s)$.

	Define $\delta_j = t(s(j))$.
	Define $\bar{\delta}_j=1-\delta_j$, then 
	$$c_1 =\delta_1 d/co+ \bar{\delta}_1d/r$$
and
$$c_{j+1} = c_j^2(\delta_{j+1} /co + \bar{\delta}_{j+1} /r)$$
for $j=1,\ldots, k-1$. 

   As $c_j(s)$ only depends on the first $j$ bits of $s$. We have
	 $c_j(s0)=c_j(s1)=c_j$ for $j=1,\ldots,k$.
	We assume $t(s)=1$, so that  $c_{k+1}(s0)= c_k^2/co$ and  $c_{k+1}(s1)= c_k^2/r$,
	and therefore we need to prove that
	\begin{equation*}
		m_{1s0} =(w_1 + c_1 + c_2+ \cdots + c_k + c_k^2/co)\cdot  d^{2^{k}}/(c_k^2/co)
	\end{equation*}
	and
	\begin{equation*}
		m_{1s1} =(m_1 + c_1 + c_2+ \cdots + c_k + c_k^2/r)\cdot  d^{2^{k}}/(c_k^2/r).
	\end{equation*}

	We now write down the deduction, then we make some remarks on how we get
	from one step to the next. We have

	\begin{align*}
		&\quad \;m_{1s0}=m_{1s}^2\\
		&=(w_1 +c_1 + \cdots + c_k)\cdot d^{2^{j-1}}/c_k\cdot
	1/c_k\cdot (m_1 +c_1 + \cdots + c_k)/d^{2^{j-1}}\\
		&=(w_1 +c_1 + \cdots + c_k)\cdot (m_1 +c_1 + \cdots + c_k)\cdot d^{2^{j}}/c_k^2\\
		&=(w_1\cdot m_1 +(c_1 + \cdots + c_k)^2 + (c_1 \bar{\delta}_1+\cdots + c_k 
		\bar{\delta}_k)\cdot (co+r))\cdot d^{2^{j}}/c_k^2\\
		&=(w_1\cdot co+t +(c_1 + \cdots + c_k)^2 + (c_1 \bar{\delta}_1+\cdots + c_k 
		\bar{\delta}_k)\cdot (co+r))\cdot d^{2^{j}}/c_k^2\\
		&=(w_1+t/co +c_1^2/co + \cdots + c_k^2/co + (c_1 \bar{\delta}_1+\cdots + c_k
		\bar{\delta}_k)\cdot (1+r/co))\cdot d^{2^{j}}/(c_k^2/co)\\
		&=(w_1 + c_1 + c_2+ \cdots + c_k + c_k^2/co)\cdot  d^{2^{j}}/(c_k^2/co).
	\end{align*}

To get from the first line to the second, we notice that, since we are 
	assuming $t(s)=1$, the exponent of $co$ in $c_k$ is odd. Therefore
	the two factors in  the expression of $m_{1s}$ (and $w_{1s}$) do not commute. But we can 
	use the identity
	$$co\cdot m_1 = w_1\cdot co$$
	to obtain
\begin{align*}
	m_{1s}&=1/c_k\cdot (m_1 +c_1 + \cdots + c_k)\cdot d^{2^{k-1}},\\
	w_{1s}&=1/c_k\cdot (w_1 +c_1 + \cdots + c_k)\cdot d^{2^{k-1}}.
\end{align*}

Then we notice that, if the exponent of $co$ in $c_j$ is even, then
	$c_jm_1+w_1c_j$ is equal to $c_j(co+r)$ by definition of $co$, and 
	$c_jm_1+w_1c_j=0$ otherwise.
That is how we get from the third line to the fourth.

To get from the sixth line to the last line, we make use of the following 
relations
	\begin{equation*}
	d/co+c_1\bar{\delta}_1 (1+r/co)=c_1,
	\end{equation*}
	and for $j=1,\cdots, k-1$,
	\begin{equation*}
		c_j^2/co+c_{j+1}\bar{\delta}_{j+1}(1+r/co)=c_{j+1};
	\end{equation*}
	the first one, for example, is proved in this way:
	\begin{align*}
		&\quad\; d/co+c_1\bar{\delta}_1 (1+r/co) \\
		&= d/co+ (\delta_1 d/co+ \bar{\delta}_1d/r)\bar{\delta}_1 (1+r/co) \\
		&= d/co+ \bar{\delta}_1 t/r (1+r/co) \\
		&=\delta_1 d/co+ \bar{\delta}_1 d/r \\
		&=c_1.
	\end{align*}

	Similarly, the following relations are used in the proof of the formula 
	involving $m_{1s1}$:
	\begin{equation*}
		d/r+c_1\delta_1 (co/r+1)=c_1,
	\end{equation*}
	and for $j=1,\cdots, k-1$,
	\begin{equation*}
		c_j^2/r+c_{j+1}{\delta}_{j+1}(co/r+1)=c_{j+1}.
	\end{equation*}
	\begin{align*}
		&\quad \;m_{1s1} =w_{1s}m_{1s}\\
		&=(m_1 +c_1 + \cdots + c_k)\cdot d^{2^{j-1}}/c_k\cdot
	1/c_k\cdot (m_1 +c_1 + \cdots + c_k)/d^{2^{j-1}}\\
		&=(m_1 +c_1 + \cdots + c_k)^2\cdot d^{2^{j}}/c_k^2\\
		&=(m_1^2 + (c_1 + \cdots + c_k)^2 + (c_1{\delta}_1+\cdots + 
		c_k{\delta}_k)\cdot (co+r))\cdot d^{2^{j}}/c_k^2\\
		&=(m_1\cdot r +d + (c_1 + \cdots + c_k)^2 + (c_1{\delta}_1+\cdots + 
		c_k{\delta}_k)\cdot (co+r))\cdot d^{2^{j}}/c_k^2\\
		&=(m_1 +d/r + (c_1 + \cdots + c_k)^2/r + (c_1{\delta}_1+\cdots + 
		c_k{\delta}_k)\cdot (co/r+1))\cdot d^{2^{j}}/(c_k^2/r)\\
		&=(m_1+c_1+c_2+\cdots + c_k + c_k^2/r)\cdot d^{2^{j}}/(c_k^2/r).\qedhere
	\end{align*}
\end{proof}
%>>>

%done do we need to suppose anything about s now? test with program: no.

%<<<
From this point on, let $s$ be a binary word of length $k$ that contains an 
even number of
$1$'s and ends with $1$. (This is the case if we set $s=\Upsilon_1\ldots
\Upsilon_n$.)

Define
	$$l=d^{2^{k-1}}/c_k=r^{2^k-1-e(s)}\cdot co^{e(s)}.$$
	Since $s$ contains an even number of $1$'s, $e(s)$ is even and $l$ is a 
	scalar.
	Define $m_0'=m_{1s(k-1)}$, and define $m_1'$,
	 $m_{1s}'$ and $r'$, $d'$, $co'$, $c_j'$
	$l'$ accordingly. 
	Because $s$ ends with $1$, we have $m_1'=m_{1s}$ and
\begin{equation*}
	m_{1ss}=m_{1s}'=(m_1' +c_1' + \cdots + c_k')\cdot l'.
\end{equation*}
We verify immediately the following relations:
	\begin{align*}
		d'&=d^{2^k}\\
		r'&=\tr(m_{1s})=l\cdot r \quad (\tr(co)=0, \text{ therefore } \tr(c_j)=0)\\
		co'&=m_1'+w_1'+r'=l\cdot co\\
		l'&= (r')^{2^k-1-e(s)}\cdot (co')^{e(s)}=l^{2^{k}}\\
		c_j'&= (d')^{2^{j-1}}/(r')^{2^j-1-e(s(j))}/(co')^{e(s(j))}
		=d^{(2^k-1)(2^{j-1})}/l^{2^j-1}\cdot c_j
	\end{align*}
	Define $l^{(0)}=l$, $L_0=1$, $c_j^{(0)}=c_j$, and $l^{(i+1)}=(l^{(i)})'$, 
	 $L_{i+1}=L_i\cdot l^{(i)}$, $c_j^{(i+1)}=(c_j^{(i)})'$ for $i\geq 0$.
	 Then $L_{i+1}=L_i^{2^k}\cdot l$.
	 We also define $L_i'$ in the obvious way. We have $L_i'=L_{i+1}/l$.
	 %>>>

	 \begin{lem}\label{lem:cjc1}
Let $s$ be a binary string of length $k$ that contains an even number of $1$'s and
ends with $1$.
		 For $j=2,\ldots, k$, for all $q\geq 0$,
		 \begin{equation}\label{eq:cjc1}
		 \frac{c_j^{(q)}/L_q}{(c_1^{(q)}/L_q)^{2^{j-1}}}= \frac{c_j}{c_1^{2^{j-1}}}
		 \end{equation}
	 \end{lem}
	 %<<<
	 \begin{proof}
		 \begin{align}\label{eq:cjc1_0}
		\frac{c_j'/l}{(c_1'/l)^{2^{j-1}}}&=
		\frac{d^{(2^k-1)(2^{j-1})}/l^{2^j}\cdot c_j}
		{(d^{(2^k-1)}/l^{2}\cdot c_1)^{2^{j-1}}}=\frac{c_j}{c_1^{2^{j-1}}}.
	\end{align}
	Therefore for $q\geq 0$,
		 \begin{equation}\label{eq:cjc1_1}
		\frac{c_j^{(q+1)}/l^{(q)}}{(c_1^{(q+1)}/l^{(q)})^{2^{j-1}}}=
		\frac{c_j^{(q)}}{(c_1^{(q)})^{2^{j-1}}}.
	\end{equation}
		 We prove \eqref{eq:cjc1} by induction. The case $q=0$ is just
		 \eqref{eq:cjc1_0}. Suppose \eqref{eq:cjc1} holds for $q$. 
		 Using $L_{q+1}=L_ql^{(q)}$, \eqref{eq:cjc1_1}, and the induction 
		 hypothesis, we find
		 \begin{equation*}
			 \quad\;\frac{c_j^{(q+1)}/L_{q+1}}{(c_1^{(q+1)}/L_{q+1})^{2^{j-1}}}
			 =\frac{c_j^{(q+1)}/L_{q}/l^{(q)}}{(c_1^{(q+1)}/L_{q}/l^{(q)})^{2^{j-1}}}
			 =\frac{c_j^{(q)}/L_q}{(c_1^{(q)}/L_q)^{2^{j-1}}}
			 = \frac{c_j}{c_1^{2^{j-1}}}.\qedhere
		 \end{equation*}
	 \end{proof}
	 %>>>

	 \begin{lem}\label{lem:c1qc1}
Let $s$ be a binary string of length $k$ that contains an even number of $1$'s and
ends with $1$.
		 For all $q\geq 0$, 
		 \begin{equation}\label{eq:c1qc1}
			 \frac{c_1^{(q+1)}/L_{q+1}}{(c_1^{(q)}/L_{q})^{2^k}}
%		 = \frac{c_1'/L_{1}}{(c_1/L_0)^{2^k}}.
			 = \frac{c_1'/l}{(c_1)^{2^k}}.
		 \end{equation}
	 \end{lem}
	 %<<<
	 \begin{proof}
		 Let $e_1=e(s(1))$.
		 $$\frac{c_1'}{c_1^{2^k}}=
		 \frac{d^{2^k}/r^{1-e_1}/co^{e_1}/l}{(d/r^{1-e_1}/co^{e_1})^{2^k}}
		 =r^{(2^k-1)(1-e_1)}co^{(2^k-1)e_1}/l.
		 $$
		 Therefore
		 $$\frac{c_1''}{(c_1')^{2^k}}
		 =(r')^{(2^k-1)(1-e_1)}(co')^{(2^k-1)e_1}/l'
		 =r^{(2^k-1)(1-e_1)}co^{(2^k-1)e_1}/l
		 = \frac{c_1'}{c_1^{2^k}},
		 $$
		 and by induction, 
		 $$\frac{c_1^{(q+1)}}{(c_1^{(q)})^{2^k}}
		 = \frac{c_1'}{c_1^{2^k}}.
		 $$

		 \begin{equation*}
			 \frac{c_1^{(q+1)}/L_{q+1}}{(c_1^{(q)}/L_{q})^{2^k}}
			 = \frac{c_1^{(q+1)}}{(c_1^{(q)})^{2^k}}\cdot \frac{L_q^{2^k}}{L_{q+1}}
			 =\frac{c_1'}{c_1^{2^k}}\cdot \frac{1}{l}. \qedhere
		 \end{equation*}
	 \end{proof}
	 %>>>
\begin{lem}\label{lem:sk}
Let $s$ be a binary string of length $k$ that contains an even number of $1$'s and
ends with $1$.
	For all $i\geq 1$,
	\begin{equation}\label{eq:grec}
		m_{1s^i}=L_i \cdot (m_1+\sum_{q=0}^{i-1}\sum_{j=1}^k c_j^{(q)}/L_q).
	\end{equation}
\end{lem}
%<<<
\begin{proof}
	When $i=1$, this is just Proposition \ref{prop:main}. Suppose \eqref{eq:grec}
	is true for $i$, then
	\begin{align*}
		&\quad \;	m_{1s^{i+1}}=m_{1s^i}'\\
		&=L_i' \cdot (m_1'+\sum_{q=0}^{i-1}\sum_{j=1}^k (c_j^{(q)})'/L_q')\\
		&=L_{i+1}/l \cdot ( (m_1+\sum_{j=1}^k c_j)\cdot l 
		+\sum_{q=0}^{i-1}\sum_{j=1}^k c_j^{(q+1)}/(L_{q+1}/l))\\
		&=L_{i+1} \cdot ( (m_1+\sum_{j=1}^k c_j  
		+\sum_{q=0}^{i-1}\sum_{j=1}^k c_j^{(q+1)}/(L_{q+1}))\\
		&=L_{i+1} \cdot ( (m_1+ 
		+\sum_{q=0}^{i}\sum_{j=1}^k c_j^{(q)}/(L_{q})).\qedhere
	\end{align*}
\end{proof}

%>>>

%<<<
Up to now, we only assumed $m_0$ and $w_0$ to be any $2\times 2$ matrix. 
From this point on, we assume
\begin{equation*}
m_0=
	\begin{pmatrix} 1& \frac{1}{u_{|u|-1}}\\ \frac{1}{u_{|u|-1}}&0\end{pmatrix}
		\begin{pmatrix} 1& \frac{1}{u_{|u|-2}}\\ \frac{1}{u_{|u|-2}}&0\end{pmatrix}
	\cdots
\begin{pmatrix} 1& \frac{1}{u_{0}}\\ \frac{1}{u_{0}}&0\end{pmatrix},
\end{equation*}
\begin{equation*}
w_0=
	\begin{pmatrix} 1& \frac{1}{v_{|u|-1}}\\ \frac{1}{v_{|u|-1}}&0\end{pmatrix}
		\begin{pmatrix} 1& \frac{1}{v_{|u|-2}}\\ \frac{1}{v_{|u|-2}}&0\end{pmatrix}
	\cdots
\begin{pmatrix} 1& \frac{1}{v_{0}}\\ \frac{1}{v_{0}}&0\end{pmatrix},
\end{equation*}
and $$s=\Upsilon_1\cdots \Upsilon_k.$$ 
By \eqref{eq:cfm} and \eqref{eq:cfmn}, 
$$\CF(\mathcal{G}(u_0,v_0,\boldsymbol \Upsilon))=\lim_{q\rightarrow\infty} 
\frac{m_{1s^q,0,1}}{m_{1s^q,0,0}}.$$
%where $s^q$ means the concatenation of $q$ $s$'s.
%>>>
\begin{lem}
	With the above definition of $m_0$, $w_0$, and $s$, 
	the sequence  $(m_{1s^q})_q$ is convergent.
\end{lem}
%<<<
\begin{proof}
	We use Lemma \ref{lem:sk}.
	We only need to
	prove the convergence of both factors on the right of \eqref{eq:grec}.
	The following facts are easy to prove:
	$$\val (m_{1,0,0})=0$$
	$$\val (m_{1,0,1}), \val (m_{1,1,0}), \val (m_{1,1,1})>0$$
	$$\val (co_{0,0})=\val (co_{1,1})=0$$
	$$\val (co_{0,1}),\val (co_{1,0})>0$$
	$$\val (r)=0$$
	$$\val (co^2)=0$$
	$$\val (d)>0$$
	$$\val ((1/co)_{0,0})=\val ((1/co)_{1,1})=0$$
	$$\val ((1/co)_{0,1}),\val ((1/co)_{1,0})>0.$$
	Therefore $\val(l)=0$, and $\val(l^{(i)})=0$ and $\val(L_i)=0$ by induction.
	And 
	\begin{equation}\label{eq:Lval}
		\val(L_{i+1}-L_i)= \val(L_i(1+l^{(i)}))=\val(l^{(i)}-1)=\val(l^{2^{ik}}-1)
		\geq 2^{ik}.
	\end{equation}
	Therefore $(L_i)_i$ converges.
	On the other hand, from Lemma \ref{lem:cjc1} and \ref{lem:c1qc1}, 
	using the facts listed above, we deduce that the valuation of each entry of 
	$c_j^{(q)}/L_q$ tends to infinity as $q$ goes up to infinity.
	\end{proof}
	%>>>

\begin{proof}[Proof of Theorem \ref{thm:pd2sum}]
	Let $f=\lim_i L_i$.
	Then $f$ is algebraic of degree at most $2^k$:
	$$f^{2^k}= \lim_i (l^{(0)}\cdots l^{(i-1)})^{2^k}= 
	\lim_i (l^{(1)}\cdots l^{(i)})=f/l.$$
	Let $H_j=\sum_{q=0}^{\infty}c_j^{(q)}/L_q$.
	Then
	$$
		\lim_{i\rightarrow \infty}m_{1s^i}=f\cdot (m_1+\sum_{j=1}^k 
		H_j).
		$$
	By Lemma \ref{lem:cjc1}, 
	$$H_j=H_1^{2^{j-1}}\cdot \frac{c_j}{c_1^{2^{j-1}}}$$
	For $j=1,\cdots, k$.
	Therefore
	$$
		\lim_{i\rightarrow \infty}m_{1s^i}=f\cdot (m_1+\sum_{j=1}^k 
H_1^{2^{j-1}}\cdot \frac{c_j}{c_1^{2^{j-1}}}.
		).
		$$
By Lemma \ref{lem:c1qc1}, 
	\begin{equation}\label{eq:H}
	H_1^{2^k}\cdot \frac{c_1'/l}{(c_1)^{2^k}}= c_1+H_1.
	\end{equation}
	If $e(s(1))=0$, then $H_1$ is a scalar matrix, and
	$$\CF(\mathcal{G}(u_0,v_0,\boldsymbol\Upsilon))\in \ff_2(A)[H_1],$$
	and is algbraic of degree at most $2^k$ over $\ff_2(A)$. 

		If $e(s(1))=1$, then for all $q$, the exponent of $co$ in $c_1^{(q)}$ is odd,
		and $H:=H_1/co$ is a scalar matrix. Insert $H_1=H\cdot co$ into 
		\eqref{eq:H} and we get
	$$
	H^{2^k}\cdot co^{2^k}\cdot \frac{c_1'/l}{(c_1)^{2^k}}= c_1+H\cdot co.
	$$
	Dividing both sides by $co$, we get an equation where the coefficients of
	$H^{2^k}$, $H$, and $H^0$ are all scalars.
	$$
	H^{2^k}\cdot co^{2^k}\cdot \frac{c_1'/co/l}{(c_1)^{2^k}}= c_1/co+H.
	$$
	Therefore $$\CF(\mathcal{G}(u_0,v_0,\boldsymbol\Upsilon))\in \ff_2(A)[H],$$
	and is algbraic of degree at most $2^k$ over $\ff_2(A)$. 
\end{proof}
%>>>

\subsection{Sums of sequences from $\mathcal{P}$ and $\mathcal{G}$}\label{ssec:pattern}
%<<<
\begin{prop}\label{prop:sump}
	Let $\mathbf{s}=\mathcal{P}(W_0,\boldsymbol\ve)$ be a binary sequence 
	in $\mathcal{P}$.
	Let $u_0=\sigma(W_{1})$, $v_0=\overline{u_0}$ (the bitwise negation of $u_0$), 
  then 
	\begin{equation*}
	\sigma(\mathbf{s})=\mathcal{G}(u_0,v_0,(1-\delta_{\ve_k,\ve_{k+1}})_k).\qedhere
	\end{equation*}
\end{prop}

%<<<
\begin{proof}
	The sequence $\bf s$ is the limit of $(W_k)_k$ defined by 
$$W_{k+1}= W_k \varepsilon_k W_k.$$ 
If the last bit of $\sigma(W_k)$ and $\epsilon_k$ have the same parity, then
$$\sigma(W_{k+1})=\sigma(W_k),\sigma(W_k);$$
otherwise
$$\sigma(W_{k+1})=\sigma(W_k),\overline{\sigma(W_k)}.$$
The last bit of $\sigma(W_{k+1})$ is always the parity of $\varepsilon_k$.
Therefore 
	$$\sigma(W_{k+2})=\sigma(W_{k+1}),\sigma(W_{k+1})$$
	if $\ve_{k+1}=\ve_k $, and 
	$$\sigma(W_{k+2})=\sigma(W_{k+1}),\overline{\sigma(W_{k+1})}$$
	otherwise.
\end{proof}
%>>>

\begin{prop}\label{prop:sumg}
	Let $\mathbf{u}=\mathcal{G}(u_0,v_0,\boldsymbol\Upsilon)$ be a binary sequence in 
	$\mathcal{G}$. Then $\sigma(\mathbf{u})=\mathcal{G}(u_2, v_2, (\Upsilon_{2+k})_k)$.
\end{prop}
%<<<
\begin{proof}
%Now suppose 
	%\begin{align}
		%u_{k+1}&=(u_k+u_k)(1-\upsilon_n)+ (u_k+v_k)\upsilon_n\\
		%v_{k+1}&=(v_k+v_k)(1-\upsilon_n)+ (v_k+u_k)\upsilon_n
%\end{align}
	%and $\mathbf{u}=\lim u_k$. 
	For a non-empty word $a$, let $a^*$ denote the word obtained from $a$
	by deleting the last bit.
	For two binary words $a$ and $b$, we have
	$$\sigma(a,b)^*=\sigma(a)^*,\sigma(b)^*$$
	if the last bit of $\sigma(a)$ is $0$, that is, if $a$ contains
	an even number of $1$'s, and
	$$\sigma(a,b)^*=\sigma(a)^*,\overline{\sigma(b)^*}$$
	otherwise.
	
	Let $N$ be an integer for which both $u_N$ and $v_N$ contains an even number 
	of $1$'s (for example $N=2$).  Then for all $k\geq 0$,
	both $u_{N+k}$ and $v_{N+k}$ contains an even number of $1$'s.
	Define $u'_k=\sigma(u_{k+N})^*$, $v'_k=\sigma(v_{k+N})^*$. 
	 If $\Upsilon_{k+N}=0$,
	$$u'_{k+1}=\sigma(u_{k+N+1})^*= \sigma(u_{k+N},u_{k+N})^*=
	\sigma(u_{k+N})^*,\sigma(u_{k+N})^* =u_k',u_k',  $$
	similarly $v'_{k+1}=v_k',v_k'$;
	 if $\Upsilon_{k+N}=1$, $u'_{k+1}=u_k',v_k'$, and $v'_{k+1}=v_k',u_k'$.
	%Therefore $\sigma(\mathbf{u})=\mathcal{G}(u_N ,v_N, (\Upsilon_{k+N})_k)$.
\end{proof}
%>>>

Corollary \ref{thm:coro} is an immediate consequence of Theorem \ref{thm:pd2sum}, Proposition \ref{prop:sump} and \ref{prop:sumg}, noticing that if 
$\boldsymbol\ve$ is a binary periodic sequence of period $n$ that contains $1$,
then $(\delta_{\ve_k,\ve_{k+1}})_k$ is periodic of period $n$ and contains
an even number of $1$'s in a period.
%>>>

\section{Acknowledgement}
The author would like to thank Jean-Paul Allouche, Alain Lasjaunias
and Zhi-Ying Wen for helpful suggestions. 

\bibliographystyle{plain}

\bibliography{article}

\begin{thebibliography}{1}

\bibitem{Bugeaud2022H}
Yann Bugeaud and Guo-Niu Han.
\newblock The {T}hue-{M}orse continued fractions in characteristic $2 $ are
  algebraic.
\newblock {\em Acta Arith.}, 2022+.

\bibitem{Hu2022}
Yining Hu.
\newblock Algebraic automatic continued fractions in characteristic~$2$.

\bibitem{Hu2022H}
Yining Hu and Guo-Niu Han.
\newblock On the algebraicity of {Thue}-{Morse} and period-doubling continued
  fractions.
\newblock {\em Acta Arith.}, 203(4):353--381, 2022.

\bibitem{Hu2022L}
Yining Hu and Alain Lasjaunias.
\newblock Period-doubling continued fractions are algebraic in characteristic
  $2$.
\newblock {\em Ann. Inst. Fourier (Grenoble)}, forthcoming.

\bibitem{Lasjaunias2017}
Alain Lasjaunias.
\newblock Continued fractions.
\newblock {\em arXiv preprint arXiv:1711.11276}, 2017.

\end{thebibliography}
\end{document}